\documentclass{article}
\usepackage[utf8]{inputenc}
\usepackage{amsmath}
\usepackage{amssymb}
\usepackage{amsthm}
\usepackage{mathrsfs}
\usepackage{siunitx}
\usepackage{enumerate}
\usepackage{url}
\usepackage[utf8]{inputenc}
\usepackage[english]{babel}
\usepackage{xcolor}
\usepackage[nottoc]{tocbibind}
\usepackage{tikz-cd}
\usepackage{stackrel}
\usepackage{geometry}[margin=1.0in]
\usepackage{stmaryrd}

\theoremstyle{plain}
\newtheorem{theorem}{Theorem}[section]
\newtheorem{corollary}[theorem]{Corollary}
\newtheorem{lemma}[theorem]{Lemma}
\newtheorem{prop}[theorem]{Proposition}

\theoremstyle{definition}
\newtheorem{definition}[theorem]{Definition}
\newtheorem{example}[theorem]{Example}
\newtheorem{remark}[theorem]{Remark}

\title{Octonion algebras over schemes and the equivalence of isotopes and isometric quadratic forms\footnote{This work was done while the author was a master student at the Department of Mathematics, Uppsala University. Supervisor was S. Alsaody.}}
\author{V. Hildebrandsson\footnote{Department of Mathematics (MAI), Linköpings University, 581 83 Linköping, Sweden. Email address: victor.hildebrandsson@gmail.com}}
\date{}

\begin{document}
\maketitle

\begin{abstract}
    Octonion algebras are certain algebras with a multiplicative quadratic form. In \cite{OA}, Alsaody and Gille show that, for octonion algebras over unital commutative rings, there is an equivalence between isotopes and isometric quadratic forms. The contravariant equivalence from unital commutative rings to affine schemes leads us to a question: can the equivalence of isometry and isotopy be generalized to octonion algebras over a (not necessarily affine) scheme? We present the basic definitions and properties of octonion algebras, both over rings and over schemes. Then we show that an isotope of an octonion algebra $\mathcal{C}$ over a scheme is isomorphic to a twist by an $\mathbf{Aut}(\mathcal{C})$–torsor. We conclude the paper by giving an affirmative answer to our question.
\end{abstract}

\tableofcontents

\section{Introduction}
The octonions $\mathbb{O}$ are a normed division algebra over $\mathbb{R}$ of dimension 8, discovered in 1843 by John Graves. Based on these, the definition of a octonion algebras over any field was given by Leonard Dickson in 1927, and this was generalized once more to octonion algebras over rings by Loos–Petersson–Racine in 2008 \cite{LPR}. An important part of these definitions is the assumption that the algebra has an associated quadratic form. In the case over a field, an octonion algebra is completely determined by its quadratic form \cite[Claim 2.3]{octalgfield}. Over rings however, this is not the case, proved by Gille in \cite{Gille_2014}. The main result of Alsaody and Gille gives a construction of the non–isomorphic octonion algebras with isometric quadratic forms. By the contravariant equivalence of the category of rings and the category of affine schemes, we expect this construction to generalize to octonion algebras over affine schemes. However, will the construction generalize to octonion algebras over general schemes?

Composition algebras over locally ringed spaces were first defined by Petersson in \cite{HP}. Since a locally ringed space $(X,\mathcal{O}_X)$ comes with a sheaf of rings $\mathcal{O}_X$, the idea is to define the composition algebra as a pair $(\mathscr{C},N)$ where $\mathscr{C}$ is an $\mathcal{O}_X$–algebra that is locally a composition algebra in the ring sense. Composition algebras only exist in ranks 1, 2, 4, 8, and a composition algebra of rank 8 is called an octonion algebra.

The structure of the article is as follows. In Section \ref{BG}, we recall background of sheaf theory over general sites, the fppf site, and basic notions of octonion algebras over rings. We also present the main result of Alsaody and Gille. In Section \ref{OoS}, we give the corresponding definitions for octonion algebras over schemes. In Section \ref{TII}, we prove several analogous results of those proved in \cite[§4 and §6]{OA}. The main results are (1) twisting an octonion algebra $\mathcal{C}$ over a scheme by an $\mathbf{Aut}(\mathcal{C})$–torsor corresponds naturally to an isotope of $\mathcal{C}$, and (2) two octonion algebras over schemes are isotopic if and only if their quadratic forms are isometric.

\subsection*{Acknowledgments}
I want to thank Susanne Pumplün for enlightening me of the work of Petersson. I also want to thank Philippe Gille for our discussion which helped improve the thesis. Lastly, I am grateful to Seidon Alsaody who introduced me to this problem and who helped me immensely, both throughout the master thesis project and when rewriting it into this article.

\section{Background}\label{BG}
In this section, we will present some background on sheaf theory on general sites and the site $(\mathbf{Sch}/X)_\text{fppf}$. We will assume knowledge in standard sheaf theory (i.e.\ over the Zariski site) and in basic schemes theory.  We will also present theory of octonion algebras over rings and the main result of Alsaody and Gille. For more on sheaf theory on general sites, see \cite[§7]{stacks}. For more on the fppf site, see \cite[§34.7]{stacks}. For more on octonion algebras over rings and proofs of results in Section \ref{OoR}, see \cite[§2 and §6]{OA}.

Throughout, if not stated otherwise, $R$ is a commutative and unital ring and $X$ is a scheme. By an \emph{$R$–ring} we mean a unital, commutative and associative $R$–algebra. We denote the category of schemes by $\mathbf{Sch}$. 

A \emph{$X$–scheme} is a scheme $Y$ together with a morphism of schemes $Y\to X$. A $X$–morphism $Y\to Z$ is a morphism of schemes such that the diagram
\[
\begin{tikzcd}
Y\arrow[rr]\arrow[dr]&&Z\arrow[dl]\\
&X
\end{tikzcd}
\]
commutes. We denote the category of $X$–schemes by $\mathbf{Sch}/X$. 

By the Yoneda lemma, we can identify a scheme $Y\in\mathbf{Sch}$ (or $X$–scheme $Y\in\mathbf{Sch}/X$) with its representable functor
\[
h_Y:\mathbf{Sch}\to\mathbf{Set}\quad(\text{or }h_Y:\mathbf{Sch}/X\to\mathbf{Set}).
\]
We use the notation $Y(T):=h_Y(T)$. 

A \emph{group scheme} (or $X$–group scheme) is a scheme (or $X$–scheme) $G$ and a factorization of its representable functor $h_G$ through the forgetful functor of groups to sets.

\subsection{Sites and sheaves}
In order to avoid set theoretical problems, we work over a fixed Grothendieck universe.

A \emph{site} is a category $\mathscr{C}$ and a class $\mathbf{Cov}(\mathscr{C})$ of families of morphisms with fixed target $\{U_i\to U\}_{i\in I}$, called coverings of $\mathscr{C}$, satisfying certain conditions, see \cite[Definition 7.6.2]{stacks}. We will denote the site by its category $\mathscr{C}$.

\begin{example}[Fppf site]
    An \emph{fppf covering} of $X$ is a family of morphisms with fixed target in $\mathbf{Sch}$, $\{p_i:X_i\to X\}_{i\in I}$, such that $X=\bigcup_{i\in I}p_i(X_i)$ and each $p_i$ is flat and of locally finite presentation. These fppf coverings  makes $\mathbf{Sch}$ into a site \cite[Lemma 34.7.3]{stacks}, denoted by $\mathbf{Sch}_\text{fppf}$. We make $\mathbf{Sch}/X$ into a site by defining 
    \[
    (\mathbf{Sch}/X)_\text{fppf}:=(\mathbf{Sch}_\text{fppf})/X.
    \]
    In other words, a covering of $Y\in\mathbf{Sch}/X$ is an fppf covering $\{Y_i\to Y\}_{i\in I}$, where $Y_i\in\mathbf{Sch}/X$, such that the diagram
    \[
    \begin{tikzcd}
        Y_i\arrow[rr]\arrow[dr]&&Y\arrow[dl]\\
        &X&
    \end{tikzcd}
    \]
    commutes for all $i\in I$.
\end{example}

Let $\mathscr{C}$ be a category. A \emph{presheaf} $\mathcal{F}$ on $\mathscr{C}$ is a contravariant functor $\mathcal{F}:\mathscr{C}\to\mathbf{Set}$. If $\mathscr{C}$ is a site, a presheaf $\mathcal{F}$ is a \emph{sheaf} if for every covering $\{U_i\to U\}_{i\in I}\in\mathbf{Cov}(\mathscr{C})$ the following holds:
\begin{enumerate}
    \item 
    if two sections $s,t\in\mathcal{F}(U)$ are such that $\forall i\in I$ $s|_{U_i}=t|_{U_i}$ then $s=t$,
    
    \item 
    if $\forall i\in I$ $s_i\in\mathcal{F}(U_i)$ are such that $\forall j\in I$ $s_i|_{U_i\times_UU_j}=s_j|_{U_i\times_UU_j}$, then there exists $s\in\mathcal{F}(U)$ such that $\forall i\in I$ $s|_{U_i}=s_i$.
\end{enumerate}
Let $\mathcal{F}$ and $\mathcal{G}$ be two presheaves. A \emph{morphism of presheaves} is a natural transformation $\eta:\mathcal{F}\to\mathcal{G}$. If $\mathcal{F}$ and $\mathcal{G}$ are sheaves, then $\eta$ is a \emph{morphism of sheaves}. The categories of presheaves and sheaves are denoted by $\mathbf{PSh}(\mathscr{C})$ and $\mathbf{Sh}(\mathscr{C})$ respectively.

Let $\mathcal{F}$ be a presheaf on the site $\mathscr{C}$. A sheafification functor $L:\mathbf{PSh}(\mathscr{C})\to\mathbf{Sh}(\mathscr{C})$ is a left adjoint to the inclusion $\mathbf{Sh}(\mathscr{C})\hookrightarrow\mathbf{PSh}(\mathscr{C})$. The \emph{sheafification of $\mathcal{F}$} is then $\mathcal{F}':=L(\mathcal{F})$ together with the canonical map $\mathcal{F}\to\mathcal{F}'$. For its construction, see \cite[§7.10]{stacks}.

Let $\mathscr{C}$ be a site and $\mathcal{G}$ be a sheaf of groups on $\mathscr{C}$. A \emph{$\mathcal{G}$–torsor} is a sheaf of sets $\mathcal{F}$ on $\mathscr{C}$ endowed with a $\mathcal{G}$–action $\mathcal{G}\times\mathcal{F}\to\mathcal{F}$ such that
\begin{enumerate}
    \item 
    whenever $\mathcal{F}(U)$ is non-empty the action $\mathcal{G}(U)\times\mathcal{F}(U)\to\mathcal{F}(U)$ is transitive,

    \item
    for every $U\in\mathscr{C}$ there exists a covering $\{U_i\to U\}_{i\in I}\in\mathbf{Cov}(\mathscr{C})$ such that $\mathcal{F}(U_i)$ is non-empty for all $i\in I$.
\end{enumerate}
A morphism of $\mathcal{G}$–torsors is a morphism of sheaves compatible with the $\mathcal{G}$–action.

\subsection{Octonion algebras over rings}\label{OoR}
Let $A=(A,\ast)$ be an algebra. It is \emph{alternative} if for each $a,b\in A$
\[
a(ab)=(aa)b\text{ and }(ba)a=b(aa).
\]
For each $a\in A$ we have the linear maps $L_a:x\mapsto a\ast x$ and $R_a:x\mapsto x\ast a$. If $A$ is alternative then $L_aR_a=R_aL_a$ for all $a\in A$, and we denote this map $B_a$.

Two $R$–algebras $(A,\ast_A)$ and $(B,\ast_B)$ are \emph{isotopic} if there exist invertible linear maps $f_i:A\to B$, $i=1,2,3$, such that
\[
f_1(x\ast_Ay)=f_2(x)\ast_Bf_3(y)
\]
for all $x,y\in A$.

\begin{remark}\label{anm3.2}
    If $A$ and $B$ are isotopic, then $A$ and $B$ are isomorphic as $R$–modules, but not necessarily as $R$–algebras. However, $f_1:A\to B$ is an isomorphism of algebras $A$ and $(B,\ast'_B)$, where
    \[
    x\ast'_By=f_2f_1^{-1}(x)\ast_Bf_3f_1^{-1}(y).
    \]
\end{remark}

An algebra $(B,\ast_B')$ is a \emph{principal isotope} of $(B,\ast_B)$ if there exist invertible linear maps $g,h:B\to B$ such that $x\ast'y=g(x)\ast h(y)$ for all $x,y\in B$. We denote $(B,\ast_B')=:B_{g,h}$. The algebras $B$ and $B_{g,h}$ are obviously isotopic. It follows from remark \ref{anm3.2} that algebras $A$ and $B$ are isotopic if and only if $A$ is isomorphic to a principal isotope of $B$.

Let $(A,\ast)$ be an $R$–algebra and $g,h:A\to A$ invertible maps. $A_{g,h}$ is unital if there exists $e\in A$ such that for all $x\in A$
\[
g(x)\ast h(e)=g(e)\ast h(x)=x.
\]
Equivalently,
\[
g^{-1}=R_{h(e)}\text{ and }h^{-1}=L_{g(e)}.
\label{eq1}\tag{$\ast$}
\]
If $A$ is a unital, alternative algebra, this implies that $g(e)$ and $h(e)$ are invertible elements (the inverse of an element is well defined in an alternative algebra). Let $a=g(e)^{-1}$ and $b=h(e)^{-1}$. Then (\ref{eq1}) is equivalent to 
\[
g=R_a \text{ and }h=L_b.
\]
Also, for any $a,b\in A^\ast$, $A_{R_a,L_b}$ is unital with unity $(ab)^{-1}$. Here $A^\ast$ denotes the invertible elements of $A$.

\begin{prop}[Alsaody and Gille, Proposition 2.1]
    Let $A$ be a unital alternative algebra over $R$ and let $A'$ be isotopic to $A$. Then $A'$ is unital if and only if $A'\simeq A_{R_a,L_b}$ for some $a,b\in A^\ast$.
\end{prop}

\begin{remark}\label{anm3.5}
    \begin{enumerate}
        \item 
        We denote $A_{R_a,L_b}=:A^{a,b}$.
        
        \item
        If $\phi:R\to S$ is a ring morphism, we can pullback an $R$–algebra $A$ to an $S$–algebra $A_S:=A\otimes_RS$. For $a,b\in A^\ast$, we have $(A^{a,b})_S=(A_S)^{a_S,b_S}$. We denote this algebra by $A^{a,b}_S$.
    \end{enumerate}
\end{remark}

A quadratic form on an $R$–module $M$ is a map $q:M\to R$ such that
\[
q(rm)=r^2q(m),\quad r\in R, m\in M,
\]
and
\[
b_q(m,n):=q(m+n)-q(m)-q(n),\quad m,n\in M
\]
is bilinear. It is multiplicative if
\[
q(mn)=q(m)q(n), \quad m,n\in M,
\]
and regular if the morphism
\begin{align*}
    \alpha:M&\to\text{Hom}(M,R)\\
    m&\mapsto b_q(m,\_) 
\end{align*}
is an isomorphism. If $M$ is finitely generated and projective, the rank of $M$ at $p\in\mathbf{Spec}(R)$ is the rank of the free $R_p$–module $M_p$. It is a locally constant function on $\mathbf{Spec}(R)$.

Now, an \emph{octonion algebra over $R$} is an $R$–algebra whose underlying module is projective of constant rank 8, and which is endowed with a regular multiplicative quadratic form. It is unital and alternative.

\begin{example}
    Over any ring $R$, we have the \emph{Zorn algebra} 
    \[
    \mathbf{Zorn}(R)=
    \begin{pmatrix}
        R&R^3\\
        R^3&R
    \end{pmatrix}
    \]
    with multiplication defined by
    \[
    \begin{pmatrix}
        a&u\\
        u'&a'
    \end{pmatrix}
    \begin{pmatrix}
        b&v\\
        v'&b'
    \end{pmatrix}=
    \begin{pmatrix}
        ab+u^tv'&av+b'u-u'\times v'\\
        a'v'+bu'+u\times v&(u')^tv+a'b'
    \end{pmatrix}
    \]
    where $u\times v$ is the vector cross product and $u^tv$ is the scalar product. With the quadratic form given by the determinant
    \[
    \det\begin{pmatrix}
        a&u\\
        u'&a'
    \end{pmatrix}
    =aa'-u^tu',
    \]
    $\mathbf{Zorn}(R)$ becomes an $R$–octonion algebra.
\end{example}

\begin{remark}
    \begin{enumerate}
        \item 
        If $(C,q)$ is an octonion algebra, we have the equality $C^\ast=\{x\in C|q(x)\in R^\ast\}$.
        
        \item
        The trace of a quadratic form $q$ is the map $tr:M\to R, m\mapsto b_q(1,m)$. Then we define the involution map by $\kappa:x\mapsto\overline{x}:=tr(x)\cdot1_C-x$.
        
        \item
        For any octonion algebra $(C,q)$ over $R$, we have the \emph{Hamilton–Cayley equation} $q(c)1=c\overline{c}$, for any $c\in C$ \cite{HP}.
    \end{enumerate}
\end{remark}

For any octonion algebra $C$ over $R$ we have the octonion sphere $\mathbf{S}_C$, which is the $R$–scheme
\[
\mathbf{S}_C(S)=\{c\in C_S|q_{C_S}(c)=1\}.
\]
We can now state the result of Alsaody and Gille.
\begin{theorem}[Alsaody and Gille, Corollary 6.7]
    Let $C$ and $C'$ be octonion algebras over $R$. The quadratic forms $q_C$ and $q_{C'}$ are isometric if and only if there exist $a,b\in\mathbf{S}_C(R)$ such that $C'\simeq C^{a,b}$.
\end{theorem}

\section{Octonion algebras over schemes}\label{OoS}
In this section, we will present the definition of an octonion algebra over a scheme and some of its constructions. These definitions will be generalizations of the ones defined in subsection \ref{OoR}.

\begin{definition}[Octonion algebra over a scheme]
    An \emph{octonion algebra} over $X$ is a tuple $(\mathcal{C},\mathcal{Q})$ consisting of
    \begin{enumerate}
        \item 
        a sheaf $\mathcal{C}$ of $\mathcal{O}_X$-algebras such that for all open affine $U\subset X$, the underlying module of the $\mathcal{O}_{X}(U)$-algebra $\mathcal{C}(U)$ is projective with constant rank 8,

        \item
        a morphism of sheaves $\mathcal{Q}:\mathcal{C}\to\mathcal{O}_X$ such that for all open affine $U\subset X$, $\mathcal{Q}(U)$ is a multiplicative, regular, quadratic form on the $\mathcal{O}_X(U)$–algebra $\mathcal{C}(U)$. We call $\mathcal{Q}$ the \emph{quadratic form on $\mathcal{C}$}.
    \end{enumerate}
\end{definition}

\begin{remark}\label{anm4.2}
    In the case where $X=\mathbf{Spec}(R)$ is an affine scheme, $\mathcal{O}_X(X)=R$, so for an octonion algebra $\mathcal{C}$ over $X$, $\mathcal{C}(X)$ is an $R$-octonion algebra (as defined in subsection \ref{OoR}).
\end{remark}

Note that defining the stalks and induced stalk maps $(\mathcal{C}_x,\mathcal{Q}_x)$ to be octonion algebras over $\mathcal{O}_{X,x}$, as defined in \cite[§1.6 and §1.7]{HP}, is another natural way to define the octonion algebra over a scheme. In the following proposition and corollary, we see that the definitions are equivalent.

\begin{prop}\label{prop3.3}
    Let $\mathcal{Q}:\mathcal{C}\to\mathcal{O}_X$ be a natural transformation. Then, for some open affine set $U\subset X$, $\mathcal{Q}(U)$ is a regular, multiplicative, and quadratic form if and only if $\mathcal{Q}_x$ is a regular, multiplicative, and quadratic form for all $x\in U$.
\end{prop}
\begin{proof}
    We only show the multiplicative part, regularity and being quadratic is similar. The induced stalk map $\mathcal{Q}_x$ at $x\in U$ can be defined in two steps:
    \begin{align*}
        1.&\quad\Tilde{\mathcal{Q}}_x:(U,s)\mapsto(U,\mathcal{Q}(U)(s))\\
        2.&\quad \mathcal{Q}_x:\overline{(U,s)}\mapsto\overline{\Tilde{\mathcal{Q}}_x(U,s)}.
    \end{align*}
    Assume $\mathcal{Q}(U)$ is a regular, multiplicative, and quadratic form. Then, by the definition of $\mathcal{Q}_x$ for $x\in U$, it follows that $\mathcal{Q}_x$ is aswell.
    
    Now, assume for all $x\in U$ that $\mathcal{Q}_x$ is multiplicative. Then $\Tilde{\mathcal{Q}}_x$ must be multiplicative. For $a,b\in\mathcal{C}(U)$, we have
    \begin{align*}
        (U,\mathcal{Q}(U)(ab))&=\Tilde{\mathcal{Q}}_x((U,ab))\\
        &=\Tilde{\mathcal{Q}}_x((U,a)(U,b))\\
        &=\Tilde{\mathcal{Q}}_x((U,a))\Tilde{\mathcal{Q}}_x((U,b))\\
        &=(U,\mathcal{Q}(U)(a))(U,\mathcal{Q}(U)(b))\\
        &=(U,\mathcal{Q}(U)(a)\mathcal{Q}(U)(b))
    \end{align*}
    so $\mathcal{Q}(U)$ must also be multiplicative. The proof of regularity and being quadratic is similar.
\end{proof}

\begin{corollary}
    For some open affine set $U\subset X$, $(\mathcal{C}(U),\mathcal{Q}(U))$ is an octonion algebra over $\mathcal{O}_X(U)$ if and only if $(\mathcal{C}_x,\mathcal{Q}_x)$ is an octonion algebra over $\mathcal{O}_{X,x}$ for all $x\in U$.
\end{corollary}
\begin{proof}
    The rank of $\mathcal{C}(U)$ at $x$ is by definition the rank of $\mathcal{C}_x$, and $\mathcal{C}(U)$ is projective if and only if $\mathcal{C}_x$ is projective for all $x\in U$ \cite[Exercise 4.11b]{CA}. Then the statement follows from proposition \ref{prop3.3}.
\end{proof}

Let $(\mathcal{C},\mathcal{Q})$ and $(\mathcal{C}',\mathcal{Q}')$ be two octonion algebras over $X$. A \emph{morphism} $\varphi:(\mathcal{C},\mathcal{Q})\to(\mathcal{C}',\mathcal{Q}')$ is a natural transformation $\varphi:\mathcal{C}\to\mathcal{C}'$ such that $\varphi(U)$ is an algebra morphism for all open sets $U\subset X$.

\begin{remark}\label{remark3.5}
    If we have an isomorphism $\varphi:(\mathcal{C}',\mathcal{Q}')\to(\mathcal{C},\mathcal{Q})$ of octonion algebras over $X$, we also have an isomorphism of octonion algebras $(\mathcal{C}(U),\mathcal{Q}(U))\simeq(\mathcal{C}'(U),\mathcal{Q}'(U))$ over the ring $\mathcal{O}_{X}(U)$. It follows that $\mathcal{Q}(U)$ and $\mathcal{Q}'(U)$ are isometric, i.e.\ 
    \[
    \mathcal{Q}'(U)=\mathcal{Q}(U)\circ\varphi(U)=(\mathcal{Q}\circ\varphi)(U).
    \] 
    Since this holds for all open affine $U\subset X$, and $X$ has an open affine cover, it follows that $\mathcal{Q}'=\mathcal{Q}\circ\varphi$, i.e.\ $\mathcal{Q}$ and $\mathcal{Q}'$ are isometric. We denote that two quadratic forms are isometric by $\mathcal{Q}\sim\mathcal{Q}'$.
\end{remark}

Let $(\mathcal{C},\mathcal{Q})$ be an octonion algebra over $X$. For an octonion algebra over a ring, the form induces a bilinear form, a trace, and a natural involution. The same goes for the quadratic form $\mathcal{Q}$. The bilinear form is the morphism of sheaves $\mathcal{B_Q}:\mathcal{C}\times\mathcal{C}\to\mathcal{O}_X$ defined for each open set $U\subset X$ by
\[
\mathcal{B_Q}(U):(c,c')\mapsto\mathcal{Q}(U)(c+c')-\mathcal{Q}(U)(c)-\mathcal{Q}(U)(c').
\]
The trace is then given by the morphism of sheaves $\text{tr}_\mathcal{C}:\mathcal{C}\to\mathcal{O}_X$ defined for each open set $U\subset X$ by
\[
\text{tr}_\mathcal{C}(U):c\mapsto\mathcal{B_Q}(U)(1,c)
\]
and the natural involution is the morphism of sheaves $\kappa_\mathcal{C}:\mathcal{C}\to\mathcal{C}$ defined for each open set $U\subset X$ by
\[
\kappa_\mathcal{C}(U):c\mapsto\text{tr}_\mathcal{C}(U)(c)\cdot1_{\mathcal{C}(U)}-c=:\overline{c}
\]
Now, let $a\in\mathcal{C}(X)$ be a global section. Let $L_a:\mathcal{C}\to\mathcal{C}$ and $R_a:\mathcal{C}\to\mathcal{C}$ be the morphisms of sheaves defined for any open $U\subset X$ by
\begin{align*}
    L_a(U):x\mapsto\text{res}_{X,U}(a)\cdot x,\quad R_a(U):x\mapsto x\cdot\text{res}_{X,U}(a).
\end{align*}
Also, let $B_a:=R_aL_a$. 

The \emph{octonion unit sphere} is the $X$–scheme defined by 
\begin{align*}
    \mathbf{S}_\mathcal{Q_C}:\mathbf{Sch}/X&\to\mathbf{Set}\\
    Y&\mapsto\{c\in \mathcal{C}_Y(Y)|\mathcal{Q}_{\mathcal{C}_Y}(Y)(c)=1\}.
\end{align*}

\begin{definition}[Isotope]
    Let $\mathcal{C}$ be an octonion algebra over $X$. For each pair $a,b\in\mathbf{S}_{\mathcal{Q_C}}(X)$, we have an \emph{isotope of} $\mathcal{C}$, denoted $\mathcal{C}^{a,b}$, which is the sheaf of $\mathcal{O}_X$-algebras such that, for any open $U\subset X$, $(\mathcal{C}^{a,b}(U),\ast)$ is the $\mathcal{O}_X(U)$-algebra with multiplication
    \[
    x\ast y:=R_a(U)(x)\cdot L_b(U)(y).
    \]
\end{definition}

\begin{remark}\label{anm4.9}
    Using the quadratic form $\mathcal{Q}$ of $\mathcal{C}$, the isotope $\mathcal{C}^{a,b}$ becomes an octonion algebra over $X$. Also, it follows that $(\mathcal{C}^{a,b}(U),\mathcal{Q}(U))$ is an isotope of $(\mathcal{C}(U),\mathcal{Q}(U))$ as $\mathcal{O}_X(U)$–octonion algebras, for all open affine $U\subset X$. By this and remark \ref{anm4.2}, we see that the definitions of an octonion algebra over a scheme and of an isotope are natural generalizations of those over rings, as we wanted.
\end{remark}
    
Now, let $Y$ be an $X$–scheme. We make an octonion algebra $\mathcal{C}$ over $X$ into an octonion algebra $\mathcal{C}_Y$ over $Y$. Consider the presheaf 
\[
V\mapsto (f^{-1}\mathcal{C})(V)\otimes_{(f^{-1}\mathcal{O}_X)(V)}\mathcal{O}_Y(V),
\]
and let $\mathcal{C}_Y$ be its sheafification (with respect to the Zariski site). Then $\mathcal{C}_Y$, together with $\mathcal{Q}_Y:=f^{-1}\mathcal{Q}\otimes id$, where $f^{-1}\mathcal{Q}$ is the inverse sheaf functor acting on the sheaf morphism $\mathcal{Q}$, is an octonion algebra over $Y$ \cite[Proposition 1.7]{HP}. (The motivation is the construction in remark \ref{anm3.5}. If $X=\mathbf{Spec}(R)$ and $Y=\mathbf{Spec}(S)$, we get exactly $\mathcal{C}_Y(Y)=C(X)_S$.)

\begin{remark}\label{anm4.3}
    \begin{enumerate}
        \item
        Let $U\subset X$ be an open set. For a section  $a\in\mathcal{C}(U)$ of an octonion algebra $\mathcal{C}$ over $X$, we have, from the definition of $\mathcal{C}_Y$, the associated section $a_Y:=[a]\otimes1\in\mathcal{C}_Y(f^{-1}(U))$, where $[a]$ is the equivalence class of $a$ in $(f^{-1}\mathcal{C})(f^{-1}(U))$. In particular, if $a\in\mathcal{C}(X)$ is a global section, $a_Y\in\mathcal{C}_Y(Y)$ is a global section.
        
        \item 
        Let $V\subset Y$ be an open set. In both $(\mathcal{C}_Y)^{a_Y,b_Y}(V)$ and $(\mathcal{C}^{a,b})_Y(V)$, the elements are of the form $c_Y$, for some $c\in \mathcal{C}(U)$, $fV\subset U\subset X$, where $U$ is open. Elements $c_Y,d_Y\in(\mathcal{C}_Y)^{a_Y,b_Y}$ are multiplied as
        \[
        c_Y\ast_1d_Y=\big(c_Y\cdot\text{res}_{Y,V}(a_Y)\big)\big(\text{res}_{Y,V}(b_Y)\cdot d_Y\big),
        \]
        while $c_Y,d_Y\in(\mathcal{C}^{a,b})_Y$ are multiplied as
        \[
        c_Y\ast_2d_Y=\big[((c\cdot\text{res}_{X,U}(a))(\text{res}_{X,U}(b)\cdot d))\big]\otimes 1.
        \]
        However,
        \begin{align*}
            c_Y\ast_2d_Y&=\big[((c\cdot\text{res}_{X,U}(a))(\text{res}_{X,U}(b)\cdot d))\big]\otimes 1\\
            &=\big[(c\cdot\text{res}_{X,U}(a))\big]\otimes1\cdot\big[(\text{res}_{X,U}(b)\cdot d)\big]\otimes1\\
            &=\big([c]\otimes1\cdot[\text{res}_{X,U}(a)]\otimes1\big)\cdot\big([\text{res}_{X,U}(b)]\otimes1\cdot[d]\otimes1\big)\\
            &=\big(c_Y\cdot\text{res}_{Y,V}(a_Y)\big)\big(\text{res}_{Y,V}(b_Y)\cdot d_Y\big)\\
            &=c_Y\ast_1d_Y,
        \end{align*}
        by naturality of the restriction maps. So $(\mathcal{C}_Y)^{a_Y,b_Y}=(\mathcal{C}^{a,b})_Y$, and we denote this sheaf by $\mathcal{C}_Y^{a,b}$.
    \end{enumerate}
\end{remark}

\begin{example}
    Recall that over $\mathbb{Z}$, we have the Zorn algebra $\mathbf{Zorn}(\mathbb{Z})$. By the contravariant equivalence of $\mathbf{Ring}$ and $\mathbf{Aff}$, $\mathbf{Zorn}(\mathbb{Z})$ defines an octonion algebra over $\mathbf{Spec}(\mathbb{Z})$ as well. Let $Y$ be a scheme. By the pullback construction, and the fact that any scheme is a $\mathbb{Z}$–scheme, $\mathbf{Zorn}(\mathbb{Z})_Y$ exists and is an octonion algebra over $Y$.
\end{example}

\section{Twists, isotopes, and isometries}\label{TII}
In this section, we will prove the generalizations of four results in \cite{OA} (theorems 4.1, 4.6, and 6.6, and corollary 6.7, corresponding to theorems 4.3, 4.7, and 4.12, and corollary 4.13 in this article). 

\subsection{The twisted octonion algebra}
Let us start with the affine case. We want to define the $R$–group schemes $\mathbf{Aut}(\mathcal{C})$, $\mathbf{SO}(q_C)$, and $\mathbf{RT}(C)$. Let $(C,q_C)$ be an octonion $R$–algebra. Denote by $\mathbf{Aut}(C)$ the automorphism $R$–group scheme of $C$, defined, for any $R$–ring $S$ by
\[
\mathbf{Aut}(C)(S):=\mathbf{Aut}(C_S),
\]
which is a closed subscheme of the $R$–group scheme $\mathbf{GL}(C)$. By $\mathbf{SO}(q_C)$ we denote the special orthogonal $R$–group scheme associated to $q_C$ \cite[Section 4.3]{calmès}. The embedding $\mathbf{Aut}(C)\to\mathbf{GL}(C)$ induces a closed immersion $\mathbf{Aut}(C)\to\mathbf{SO}(q_C)$. Define $\mathbf{RT}(C)$ as the closed $R$–subscheme of $\mathbf{SO}(q_C)^3$ such that for any $R$–ring $S$ we have
\[
\mathbf{RT}(C)(S)=\left\{(t_1,t_2,t_3)\in\mathbf{SO}(q_C)^3|t_{1}(xy)=\overline{t_{2}(\overline{x})}\cdot\overline{t_{3}(\overline{y})}\text{ for any }x,y\in C_S\right\}.
\]
Now we can continue with octonion algebras over general schemes. Let $X$ be a scheme. From here onwards, unless otherwise stated, we denote an $X$–scheme $Y\overset{f}{\to}X$ by its scheme $Y$. Let $(\mathcal{C},\mathcal{Q_C})$ be an octonion algebra over $X$. Let $\mathbf{Aut}(\mathcal{C})$ be the affine group scheme over $X$ such that
\[
\mathbf{Aut}(\mathcal{C})(Y):=\mathbf{Aut}(\mathcal{C}_Y).
\]
The $X$–group scheme $\mathbf{O}(\mathcal{Q}_\mathcal{C})$ is defined by
\begin{align*}
    \mathbf{O}(\mathcal{Q}_\mathcal{C})(Y):=\{\varphi:\mathcal{C}_Y\to\mathcal{C}_Y|&\varphi_Z:(\mathcal{C}_Y)_Z\to(\mathcal{C}_Y)_Z\text{ is linear, invertible,}\\
    &\text{and }\mathcal{Q}_{(\mathcal{C}_Y)_Z}\circ\varphi=\mathcal{Q}_{(\mathcal{C}_Y)_Z}\text{ for any $Y$–scheme }Z\}
\end{align*}
\cite[Section 4.1]{calmès} and, similar to the affine case, $\mathbf{SO}(\mathcal{Q}_\mathcal{C})$ is the special orthogonal $X$–group scheme \cite[Section 4.3]{calmès}. Recall that for $c\in\mathcal{C}(U)$, we have the natural involution $\kappa_\mathcal{C}(U)(c)=\text{tr}_\mathcal{C}(U)(c)-c$. We then define $\mathbf{RT}(\mathcal{C})$ as the closed $X$–subscheme of $\mathbf{SO}(\mathcal{Q}_\mathcal{C})^3$ defined by
\begin{align*}
    \mathbf{RT}(\mathcal{C})(Y):=\{
    (t_1,t_2,t_3)\in\mathbf{SO}&(\mathcal{Q}_\mathcal{C})(Y)^3\big|\text{ for any $Y$–scheme }Z\\
    &t_{1Z}(Z)(xy)=(\kappa_{(\mathcal{C}_Y)_Z} t_{2Z}\kappa_{(\mathcal{C}_Y)_Z})(Z)(x)\cdot(\kappa_{(\mathcal{C}_Y)_Z} t_{3Z}\kappa_{(\mathcal{C}_Y)_Z})(Z)(y)\}.
\end{align*}
where $t_{iZ}$, $i=1,2,3$, are induced by the $Y$–scheme structure morphism $Z\to Y$.

\begin{remark}\label{anm4.1}
    These schemes can equivalently be defined in with a global–local point of view:
    \[
    \mathbf{O}(\mathcal{Q}_\mathcal{C})(Y)=\{\varphi:\mathcal{C}_Y\to\mathcal{C}_Y|\text{ for all open affine }V\subset Y; \varphi(V)\in\mathbf{O}(\mathcal{Q}_{\mathcal{C}_Y}(V))\},
    \]
    \[
    \mathbf{SO}(\mathcal{Q}_\mathcal{C})(Y)=\{\varphi\in\mathbf{O}(\mathcal{Q}_\mathcal{C})(Y)|\text{ for all open affine }V\subset Y; \varphi(V)\in\mathbf{SO}(\mathcal{Q}_{\mathcal{C}_Y}(V))\},
    \]
    \begin{align*}
        \mathbf{RT}(\mathcal{C})(Y)=\{
        (t_1,t_2,t_3)\in\mathbf{SO}(\mathcal{Q}_\mathcal{C})(Y)^3|&\text{ for all open affine }V\subset Y\\
        &(t_1(V),t_2(V),t_3(V))\in\mathbf{RT}(\mathcal{C}_Y(V))(\mathcal{O}_Y(V))\},
    \end{align*}
    From \cite{OA}, we know many properties of these group schemes over rings ($\simeq$ affine schemes). We will see that this perspective will prove useful.
\end{remark}
\begin{example}\label{ex5.2}
    We see that, for any affine open $U\subset X$, $a\in\mathcal{C}(X)$, $x,y\in\mathcal{C}(U)$, the triple $(B_a,R_{\overline{a}},L_{\overline{a}})$ satisfies
    \begin{align*}
        B_a(U)(xy)&=R_a(U)L_a(U)(xy)\\
        &=\text{res}_{X,U}(a)\cdot ((xy)\cdot\text{res}_{X,U}(a))\\
        &=(\overline{\text{res}_{X,U}(\overline{a})\cdot\overline{x}})\cdot(\overline{\overline{y}\cdot\text{res}_{X,U}(\overline{a})})\\
        &=R_{\overline{a}}(U)(x)L_{\overline{a}}(U)(y),
    \end{align*}
    where the third equality follows from the fourth Moufang identity. Hence $(B_a,R_{\overline{a}},L_{\overline{a}})\in\mathbf{RT}(\mathcal{C})(U)$.
\end{example}

We have a natural action of $\mathbf{RT}(\mathcal{C})$ on $\mathbf{S}_\mathcal{Q_C}\times_X\mathbf{S}_\mathcal{Q_C}=:\mathbf{S}_\mathcal{Q_C}^2$ by
\[
(t_1,t_2,t_3).(u,v)=(t_3(Y)(u),t_2(Y)(v))
\]
for any $(t_1,t_2,t_3)\in\mathbf{RT}(\mathcal{C})(Y)$ and any $u,v\in\mathbf{S}_\mathcal{Q_C}(Y)$. We also have a map $\Pi:\mathbf{RT}(\mathcal{C})\to\mathbf{S}_\mathcal{Q_C}^2$, defined for any $\mathbf{t}:=(t_1,t_2,t_3)\in\mathbf{RT}(\mathcal{C})(Y)$ by
\[
\Pi_Y: \mathbf{t}\mapsto(t_3(Y)(1),t_2(Y)(1)).
\]
\begin{theorem}\label{thm4.10}
\begin{enumerate}
    \item 
    The stabilizer $\text{\emph{Stab}}_{\mathbf{RT}(\mathcal{C})}(1,1)$ is $i(\mathbf{Aut}(\mathcal{C}))$, where $i:\mathbf{Aut}(\mathcal{C})\hookrightarrow\mathbf{RT}(\mathcal{C})$ is the natural transformation $i_Y:t\mapsto(t,t,t)$.
    
    \item
    The fppf quotient sheaf $\mathbf{RT}(\mathcal{C})/\mathbf{Aut}(\mathcal{C})$ is representable by an $X$–scheme and the map
    \[
    \mathbf{RT}(\mathcal{C})/\mathbf{Aut}(\mathcal{C})\to\mathbf{S}_\mathcal{Q_C}^2,
    \]
    induced by $\Pi$, is an $X$–isomorphism.
\end{enumerate}
\end{theorem}
\begin{proof}
    \begin{enumerate}
        \item 
        First we have to show that $i$ is well-defined. Take $t\in\mathbf{SO}(\mathcal{Q}_\mathcal{C})(Y)$ and open affine $V\subset Y$. Then $t(V)$ is an automorphism if and only if \[
        (t(V),t(V),t(V))\in\mathbf{RT}(\mathcal{C}_Y(V))(\mathcal{O}_Y(V))\]
        \cite[Lemma 3.5]{OA}. Hence $t\in\mathbf{Aut}(\mathcal{C})(Y)$ if and only if $(t,t,t)\in\mathbf{RT}(\mathcal{C})(Y)$.
        
        Now, $(t_1(V),t_2(V),t_3(V))\in\mathbf{RT}(\mathcal{C}_Y(V))(\mathcal{O}_Y(V))$ satisfies 
        \[
        t_3(V)(1)=t_2(V)(1)=1
        \]
        if and only if
        \[
        (t_1(V),t_2(V),t_3(V))\in i_V\big(\mathbf{Aut}(\mathcal{C}_Y(V))\big)
        \]
        \cite[Proposition 3.8]{OA}. Since it holds for all open affine $V\subset Y$, and $Y$ can be covered by open affines, we have that Stab$_{\mathbf{RT}(\mathcal{C})}(1,1)=i(\mathbf{Aut}(\mathcal{C}))$. 
        
        \item
        We first want to show that $\mathbf{Aut}(\mathcal{C})$ is flat over $X$. Let $h:\mathbf{Aut}(\mathcal{C})\to X$ be the canonical map, and $X=\bigcup_iU_i$ an affine cover, $U_i=\mathbf{Spec}(R_i)$ (note that, by \cite[Lemma 34.7.2]{stacks}, this cover is also an fppf covering). Since $\mathbf{Aut}(\mathcal{C})$ is an affine $X$-scheme, $V_i:=h^{-1}(U_i)$ is an affine open set. $\mathbf{Aut}(\mathcal{C})|_{U_i}$ is an $R_i$–scheme, and, by definition, $V_i$ is an open set of $\mathbf{Aut}(\mathcal{C})|_{U_i}$. Then $\mathcal{O}_{\mathbf{Aut}(\mathcal{C})|_{U_i}}(V_i)$ is a flat $R_i$–module \cite[Corollary 4.12]{LPR}. But $\mathcal{O}_{\mathbf{Aut}(\mathcal{C})|_{U_i}}(V_i)=\mathcal{O}_{\mathbf{Aut}(\mathcal{C})}(V_i)$, so $h$ is a flat morphism \cite[Lemma 29.25.3]{stacks}. Then $\mathbf{Aut}(\mathcal{C})$ is flat over $X$. It follows that $\mathbf{RT}(\mathcal{C})/\mathbf{Aut}(\mathcal{C})$ is representable by an $X$–scheme, and the induced map is a monomorphism \cite[Expose XVI Theorem 2.2]{SGA}.
        
        Let
        \begin{align*}
            P:\mathbf{Sch}/X&\to\mathbf{Set}\\
            Z&\mapsto\mathbf{RT}(\mathcal{C})(Z)/\mathbf{Aut}(\mathcal{C})(Z)
        \end{align*}
        be a presheaf. Then $\widetilde{P}:=\mathbf{RT}(\mathcal{C})/\mathbf{Aut}(\mathcal{C})$ is its the (fppf) sheafification of $P$. Let $\{U_i\}_{i\in I}$ be an affine cover of $X$ and consider the presheaf
        \begin{align*}
            Q_i:\mathbf{Sch}/U_i&\to\mathbf{Set}\\
            Z&\mapsto\mathbf{RT}(\mathcal{C})|_{U_i}(Z)/\mathbf{Aut}(\mathcal{C})|_{U_i}(Z)
        \end{align*}
        Then $\widetilde{Q_i}:=\mathbf{RT}(\mathcal{C})|_{U_i}/\mathbf{Aut}(\mathcal{C})|_{U_i}$ is its (fppf) sheafification. For any $Z\in\mathbf{Sch}/U_i\subset\mathbf{Sch}/X$, the presheaves $P$ and $Q_i$ coincide:
        \[
        P(Z)=\mathbf{RT}(\mathcal{C})(Z)/\mathbf{Aut}(\mathcal{C})(Z)=\mathbf{RT}(\mathcal{C})|_{U_i}(Z)/\mathbf{Aut}(\mathcal{C})|_{U_i}(Z)=Q_i(Z).
        \]
        Then both $\widetilde{Q_i}$ and $\Big(\mathbf{RT}(\mathcal{C})/\mathbf{Aut}(\mathcal{C})\Big)\Big|_{U_i}$ satisfy the universal property of sheafification of $P|_{U_i}$. It follows from uniqueness that
        \[
        \Big(\mathbf{RT}(\mathcal{C})/\mathbf{Aut}(\mathcal{C})\Big)\Big|_{U_i}\simeq\widetilde{Q_i}=\mathbf{RT}(\mathcal{C})|_{U_i}/\mathbf{Aut}(\mathcal{C})|_{U_i}.
        \]
        Note also that, for any $Z\in\mathbf{Sch}/U_i\subset\mathbf{Sch}/X$, we also have the equivalence
        \begin{align*}
            (\mathbf{S}_\mathcal{Q_C}|_{U_i}\times_{U_i}\mathbf{S}_\mathcal{Q_C}|_{U_i})(Z)&=\mathbf{S}_\mathcal{Q_C}|_{U_i}(Z)\times_{U_i}\mathbf{S}_\mathcal{Q_C}|_{U_i}(Z)\\
            &= \mathbf{S}_\mathcal{Q_C}(Z)\times_X\mathbf{S}_\mathcal{Q_C}(Z)\\
            &=(\mathbf{S}_\mathcal{Q_C}\times_X\mathbf{S}_\mathcal{Q_C})(Z)\\
            &=(\mathbf{S}_\mathcal{Q_C}\times_X\mathbf{S}_\mathcal{Q_C})|_{U_i}(Z).
        \end{align*}
        Now, restricting the induced monomorphism we get
        \[
        \mathbf{RT}(\mathcal{C})|_{U_i}/\mathbf{Aut}(\mathcal{C})|_{U_i}\to\big(\mathbf{S}_\mathcal{Q_C}|_{U_i}\big)^2,
        \]
        which is an $U_i$–isomorphism \cite[Theorem 4.1(2)]{OA}. It follows that
        \[
        \mathbf{RT}(\mathcal{C})/\mathbf{Aut}(\mathcal{C})\overset{\sim}{\to}\mathbf{S}_\mathcal{Q_C}^2
        \]
        is an $X$–isomorphism.
    \end{enumerate}
\end{proof}

Let $X=\bigcup_iU_i$ be an affine cover and $\mathbf{G}:=\mathbf{Aut}(\mathcal{C})$. It follows from theorem \ref{thm4.10} and \cite[Corollary III.4.1.8]{GA} that $\Pi^{-1}|_{U_i}$ is a $\mathbf{G}|_{U_i}$–torsor in the fppf topology. Then $\Pi^{-1}$ is a $\mathbf{G}$–torsor over $\mathbf{S}_\mathcal{Q_C}^2$ in the fppf topology.

We have the additive $X$–group scheme $\mathbf{W}(\mathcal{C})$ defined by
\[
\mathbf{W}(\mathcal{C})(Y)=\mathcal{C}_Y(Y).
\]
Given $a,b\in\mathbf{S}_\mathcal{Q_C}(Y)$, consider the $\mathbf{G}$–torsor $\mathbf{E}^{a,b}:=\Pi^{-1}(a,b)$ over $X$. Let the presheaf $\mathbf{P}^{a,b}(\mathcal{C})$ of $X$–schemes be definied by
\[
\mathbf{P}^{a,b}(\mathcal{C})(Y)=(\Pi^{-1}_Y(a_Y,b_Y)\times \mathcal{C}_Y(Y))/\sim
\]
where
\[
(\mathbf{t},x)\sim(\mathbf{t}',x')\iff\exists g\in\mathbf{G}(Y):(\mathbf{t}g,x)=(\mathbf{t}',g(Y)(x')),
\]
and let $\mathbf{E}^{a,b}\wedge^\mathbf{G} \mathbf{W}(\mathcal{C})$ be its sheafification.

\begin{remark}\label{anm4.10}
    Let $\Pi_Y(\mathbf{t})=\Pi_Y(\mathbf{t}')$. Then
    \begin{align*}
        t_3(Y)(1)&=t_3'(Y)(1)\iff t_3^{-1}(Y)t_3'(Y)(1)=1\iff(t_3^{-1}t_3')(Y)(1)=1\\
        t_2(Y)(1)&=t_2'(Y)(1)\iff t_2^{-1}(Y)t_2'(Y)(1)=1\iff(t_2^{-1}t_2')(Y)(1)=1
    \end{align*}
    which implies that $\mathbf{t}^{-1}\mathbf{t}'\in\mathbf{Stab}_{\mathbf{RT}(\mathcal{C})}(1,1)(Y)=i_Y(\mathbf{G}(Y))$.
\end{remark}

The presheaf $\mathbf{P}^{a,b}(\mathcal{C})$ is in fact a presheaf of algebras; the linear structure on $\mathbf{P}^{a,b}(\mathcal{C})(Y)$ is induced by 
\[
(\mathbf{t},x)+(\mathbf{t}',x')=(\mathbf{t},x+(t_1^{-1}t_1')(Y)(x'))
\]
and, for $s\in\mathcal{O}_Y(Y)$,
\[
s(\mathbf{t},x)=(\mathbf{t},sx),
\]
and the multiplication
\[
(\mathbf{t},x)(\mathbf{t}',x')=(\mathbf{t},x(t_1^{-1}t_1')(Y)(x')).
\]
This makes $\mathbf{P}^{a,b}(\mathcal{C})(Y)$ into an $\mathcal{O}_Y(Y)$–algebra. The unity is the class of $(\mathbf{t},1)$, indepentent of choice of $\mathbf{t}\in\Pi^{-1}_Y(a_Y,b_Y)$ by remark \ref{anm4.10}. From this and the fact that Cayley–Hamilton equation holds on $\mathcal{Q}_Y(V)$ for any open $V\subset Y$, it follows that $\mathbf{P}^{a,b}(\mathcal{C})(Y)$ has the same quadratic form as $\mathcal{C}_Y$ \cite[p.885]{OA}. So $\mathbf{P}^{a,b}(\mathcal{C})(Y)$ is an $\mathcal{O}_Y(Y)$–octonion algebra.

\begin{lemma}\label{prop4.11}
    Let $a,b\in\mathbf{S}_\mathcal{Q_C}(X)$ and $(t_1,t_2,t_3)\in\mathbf{RT}(\mathcal{C})(Y)$. Then $t_1(Y)$ is an isomorphism $\mathcal{C}_Y(Y)\to\mathcal{C}_Y^{a,b}(Y)$ if and only if $t_2(Y)(1)=\eta b_Y$ and $t_3(Y)(1)=\eta a_Y$ for some $\eta\in\boldsymbol{\boldsymbol{\mu}}_2(Y)$.
\end{lemma}
\begin{proof}
    By definition, $\mathcal{C}_Y(Y)$ and $\mathcal{C}_Y^{a,b}(Y)$ are $\mathcal{O}_Y(Y)$–octonion algebras, and $t_1(Y)$ an algebra homomorphism between them. The statement then follows from \cite[Proposition 4.5]{OA}.
\end{proof}

\begin{lemma}\label{lemma4.12}
Let $a,b\in\mathbf{S}_\mathcal{Q_C}(X)$. For each $Y\in\mathbf{Sch}/X$ such that $\Pi^{-1}_Y(a_Y,b_Y)\neq\varnothing$, the map
\begin{align*}
    \Omega_Y^{a,b}:\mathbf{P}^{a,b}(\mathcal{C})(Y)&\to\mathcal{C}_Y^{a,b}(Y)\\
    (\mathbf{t},x)&\mapsto t_1(Y)(x)
\end{align*}
is an algebra isomorphism.
\end{lemma}
\begin{proof}
    The proof is similar to that of \cite[Theorem 4.6]{OA}.
    
    The map $\Omega^{a,b}_Y$ is well–defined for any $X$–scheme $Y$ since if $(\mathbf{t},x)\sim(\mathbf{t}',x')$ then $x'=g(Y)^{-1}(x)$ and $\mathbf{t}'=\mathbf{t}g$, for some $g\in\mathbf{G}(Y)$, so 
    \[
    t_1'(Y)(x')=(t_1g)(Y)(g(Y)^{-1}(x))=t_1(Y)g(Y)g(Y)^{-1}(x)=t_1(Y)(x).
    \]
    $\Omega^{a,b}(Y)$ is linear since,
    \[
    \Omega^{a,b}_Y(s(\mathbf{t},x))=\Omega^{a,b}_Y((\mathbf{t},sx))=t_1(Y)(sx)=st_1(Y)(x)=s\Omega^{a,b}_Y(\mathbf{t},x),
    \]
    for $s\in\mathcal{O}_Y(Y)$, and 
    \begin{align*}
        \Omega^{a,b}_Y((\mathbf{t},x)+(\mathbf{t}',x'))&=\Omega^{a,b}_Y(\mathbf{t},x+(t_1^{-1}t_1')(Y)(x'))\\
        &=t_1(Y)(x+(t_1^{-1}t_1')(Y)(x'))\\
        &=t_1(Y)(x)+t_1(Y)((t_1^{-1}t_1')(Y)(x'))\\
        &=t_1(Y)(x)+t_1'(Y)(x')\\
        &=\Omega^{a,b}_Y(\mathbf{t},x)+\Omega^{a,b}_Y(\mathbf{t}',x').
    \end{align*}
    We have
    \[
    \Omega^{a,b}_Y((\mathbf{t},x)(\mathbf{t}',x'))=\Omega^{a,b}_Y(\mathbf{t},x(t_1^{-1}t_1')(Y)(x'))=t_1(Y)(x(t_1^{-1}t_1')(Y)(x')).
    \]
    Since $\mathbf{t}\in\Pi^{-1}_Y(a_Y,b_Y)$, it follows from lemma \ref{prop4.11} that this is equal to
    \begin{align*}
        t_1(Y)(x)\ast_{a_Y,b_Y}t_1(Y)((t_1^{-1}t_1')(Y)(x'))&=t_1(Y)(x)\ast_{a_Y,b_Y}t_1'(Y)(x')\\
        &=\Omega^{a,b}_Y(\mathbf{t},x)\ast_{a_Y,b_Y}\Omega^{a,b}_Y(\mathbf{t}',x')
    \end{align*}
    If $\Omega^{a,b}_Y(\mathbf{t},x)=t_1(Y)(x)=0$ then $x=0$ since $t_1(Y)$ is invertible. Then 
    \[
    \ker\Omega^{a,b}_Y=\{(\mathbf{t},0)|\mathbf{t}\in\Pi^{-1}_Y(a_y,b_Y)\}.
    \]
    By definition, for any $(\mathbf{t},0),(\mathbf{t}',0)\in\ker\Omega^{a,b}_Y$, $\Pi_Y(\mathbf{t})=\Pi_Y(\mathbf{t}')$. From remark \ref{anm4.10}, it follows that $\mathbf{t}^{-1}\mathbf{t}'\in i_Y(\mathbf{G}(Y))$. Then
    \[
    (\mathbf{t}(\mathbf{t}^{-1}\mathbf{t}'),(\mathbf{t}^{-1}\mathbf{t}')(Y)(0))=(\mathbf{t}',0)
    \]
    so $(\mathbf{t},0)\sim(\mathbf{t}',0)$. Then $\ker\Omega^{a,b}_Y=0$ and $\Omega^{a,b}_Y$ is injective. Given $x\in\mathcal{C}^{a,b}_Y(Y)$ then any $\mathbf{t}\in\Pi^{-1}_Y(a,b)\neq\varnothing$ satisfy $\Omega^{a,b}_Y(\mathbf{t},t^{-1}_1(Y)(x))=x$, so $\Omega^{a,b}_Y$ is surjective.
\end{proof}

Now we can show the link between twisted octonion algebras and isotopes.

\begin{theorem}\label{thm4.7}
Consider the algebra isomorphisms $\Omega^{a,b}_Y$, defined in lemma \ref{lemma4.12}. The induced map 
\[
\widehat{\Omega}^{a,b}:\mathbf{E}^{a,b}\wedge^\mathbf{G} \mathbf{W}(\mathcal{C})\overset{\sim}{\to}\mathbf{W}(\mathcal{C}^{a,b}).
\]
is a natural isomorphism of fppf-sheaves of algebras.
\end{theorem}

\begin{remark}
    Naturality means that for any morphism $g:Z\to Y$ of $X$–schemes, the diagram 
    \[
    \begin{tikzcd}
    \mathbf{P}^{a,b}(\mathcal{C})(Z)\arrow[r,"\Omega^{a,b}_Z"]&\mathcal{C}^{a,b}_Z(Z)\\
    \mathbf{P}^{a,b}(\mathcal{C})(Y)\arrow[r,"\Omega^{a,b}_Y"]\arrow[u,"\widetilde{g}:=\mathbf{P}^{a,b}(\mathcal{C})(g)"]&\mathcal{C}^{a,b}_Y(Y)\arrow[u,swap,"\widehat{g}:=\mathcal{C}^{a,b}_Y(g)"]
    \end{tikzcd}
    \]
    commutes. The maps are well-defined by remark \ref{anm4.3}.
\end{remark}

\begin{proof}
    We first prove the isomorphism. It follows from the universal property of sheafification that if the map $\Omega_Y^{a,b}$ on $\mathbf{P}^{a,b}(\mathcal{C})(Y)$ is an algebra isomorphism then $\widehat{\Omega}^{a,b}_Y$ is also an algebra isomorphism. Since $\Omega^{a,b}$ is globally defined, it suffices to check an open cover of $Y$. The (fppf) sheaf $\mathbf{E}^{a,b}\wedge^\mathbf{G}\mathbf{W}(\mathcal{C})$ is a $\mathbf{G}$–torsor so for every $X$–scheme $Y$ there exists an fppf covering $\{Y_i\to Y\}_{i\in I}$ such that $\mathbf{P}^{a,b}(\mathcal{C})(Y_i)\neq\varnothing$. Hence we only need to check $Y$ such that $\mathbf{P}^{a,b}(\mathcal{C})(Y)\neq\varnothing$. Let $Y$ be such a scheme. It follows from lemma \ref{lemma4.12} that $\Omega^{a,b}_Y$ is an algebra isomorphism. 
    
    Now we prove naturality. It suffices to check locally. For any open affine $U\subset X$ and any morphism $h:W\to V$ of $U$–schemes, the diagram
    \[
    \begin{tikzcd}
    \mathbf{P}^{a,b}(\mathcal{C})|_U(W)\arrow[r,"\Omega^{a,b}_W"]&(\mathcal{C}|_U)^{a,b}_W(W)\\
    \mathbf{P}^{a,b}(\mathcal{C})|_U(V)\arrow[r,"\Omega^{a,b}_V"]\arrow[u,"\widetilde{h}"]&(\mathcal{C}|_U)^{a,b}_V(V)\arrow[u,"\widehat{h}"]
    \end{tikzcd}
    \]
    commutes \cite[Theorem 4.6]{OA}. 
\end{proof}

\subsection{The equivalence of isotopes and isometries}
Consider the site $(\mathbf{Sch}/X)_\text{fppf}$. If $\mathbf{G}$ is a sheaf of groups in $(\mathbf{Sch}/X)_\text{fppf}$, we denote the set of isomorphism classes of $\mathbf{G}$–torsors, called the \emph{first cohomology}, by $H^1_\text{fppf}(X,\mathbf{G})$ (see \cite[§2.2]{Gille}).

\begin{lemma}\label{lemma4.9}
    Let $(\mathcal{C},\mathcal{Q_C})$ be an octonion algebra over $X$. We have bijections of pointed sets 
    \begin{align*}
        H^1_\text{fppf}(X,\mathbf{Aut}(\mathcal{C}))&\simeq\{C'|C'\simeq\mathcal{C}\text{ fppf locally}\}\\
        H^1_\text{fppf}(X,\mathbf{O}(\mathcal{Q_C}))&\simeq\{(\mathcal{M},\mathcal{Q_M})|\mathcal{Q_M}\sim\mathcal{Q_C}\text{ fppf locally}\}
    \end{align*}
    with base points $\mathcal{C}$ and $(\mathcal{C},\mathcal{Q_C})$ respectively, which preserves the base points. (Here $C'$ are octonion algebras over $X$ and $(\mathcal{M},\mathcal{Q_M})$ are $\mathcal{O}_X$–module sheaves with a quadratic form.)
\end{lemma}
\begin{proof}
    We have mutually inverse functions
    \begin{align*}
        H^1_\text{fppf}(X,\mathbf{Aut}(\mathcal{C}))&\leftrightarrow\{C'|C'\simeq\mathcal{C}\text{ fppf locally}\}\\
        E&\mapsto E\wedge^{\mathbf{Aut}(\mathcal{C})}\mathcal{C}\\
        \mathbf{Isom}(\mathcal{C},C')&\mapsfrom C'
    \end{align*}
    where $E\wedge^{\mathbf{Aut}(\mathcal{C})}\mathcal{C}$ is the sheafification of
    \[
    \mathbf{E}(\mathcal{C})(Y)=(E(Y)\times \mathcal{C}_Y(Y))/\sim
    \]
    where
    \[
    (u,x)\sim(u',x')\iff\exists g\in\mathbf{Aut}(\mathcal{C})(Y):(ug,x)=(u',g(Y)(x')),
    \]
    and $\mathbf{Isom}(\mathcal{C},C')$ is the set of isomorphisms from $\mathcal{C}$ to $C$'. The other bijection is similar.
\end{proof}

Recall the scheme $\boldsymbol{\mu}_2$. We can embed this into $\mathbf{RT}(\mathcal{C})$ by
\begin{align*}
    \boldsymbol{\mu}_2(Y)&\hookrightarrow\mathbf{RT}(\mathcal{C})(Y)\\
    \eta&\mapsto(1,\eta,\eta)
\end{align*}
We also have an embedding into $\mathbf{S}_\mathcal{Q_C}^2$:
\begin{align*}
    \boldsymbol{\mu}_2(Y)&\hookrightarrow\mathbf{S}_\mathcal{Q_C}^2(Y)\\
    \eta&\mapsto(\eta\cdot1_\mathcal{C},\eta\cdot1_\mathcal{C})
\end{align*}
We can now mod out $\boldsymbol{\mu}_2$ from the mapping $\Pi$, and get
\[
\Pi_+:\mathbf{RT}(\mathcal{C})/\boldsymbol{\mu}_2\to\mathbf{S}_\mathcal{Q_C}^2/\boldsymbol{\mu}_2
\]
Let $\{U_i\}$ be an affine cover of $X$ and $f_1:\mathbf{RT}(\mathcal{C})\to\mathbf{SO}(\mathcal{Q_C})$ be the projection onto the first coordinate. The map $f_1$ induces an $U_i$–isomorphism 
\[
\big(\mathbf{RT}(\mathcal{C})/\boldsymbol{\mu}_2\big)\big|_{U_i}\simeq\mathbf{RT}(\mathcal{C})|_{U_i}\big/\boldsymbol{\mu}_2|_{U_i}\overset{\sim}{\longrightarrow}\mathbf{SO}(\mathcal{Q_C})|_{U_i}
\]
\cite[p. 893]{OA}, where the first isomorphism comes from a similar argument as for the quotient scheme in theorem \ref{thm4.10}. It follows that we have an $X$–isomorphism $\mathbf{RT}(\mathcal{C})/\boldsymbol{\mu}_2\overset{\sim}{\longrightarrow}\mathbf{SO}(\mathcal{Q_C})$, so we get a commutative diagram
\[
\begin{tikzcd}
    \mathbf{RT}(\mathcal{C})\arrow[r,"f_1"]\arrow[d,"\Pi"]&\mathbf{SO}(\mathcal{Q_C})\arrow[d,"\Pi_+"]\\
    \mathbf{S}_\mathcal{Q_C}^2\arrow[r,"\rho"]&\mathbf{S}_\mathcal{Q_C}^2/\boldsymbol{\mu}_2
\end{tikzcd}
\label{cartdiag}\tag{$\dag$}
\]
where $\Pi$ and $\Pi_+$ are $\mathbf{Aut}(\mathcal{C})$–torsors.
\begin{lemma}
    The diagram $($\ref{cartdiag}$)$ is a Cartesian square.
\end{lemma}
\begin{proof}
    Let $\{U_i\}_{i\in I}$ be an affine open cover of $X$. For each $i\in I$, the restricted diagram
    \[
    \begin{tikzcd}
        \mathbf{RT}(\mathcal{C})|_{U_i}\arrow[r,"f_1|_{U_i}"]\arrow[d,"\Pi|_{U_i}"]&\mathbf{SO}(\mathcal{Q_C})|_{U_i}\arrow[d,"\Pi_+|_{U_i}"]\\
        \mathbf{S}_\mathcal{Q_C}^2|_{U_i}\arrow[r,"\rho|_{U_i}"]&(\mathbf{S}_\mathcal{Q_C}^2/\boldsymbol{\mu}_2)|_{U_i}
    \end{tikzcd}
    \]
    is a Cartesian square \cite[p. 893]{OA}. Let $Y\in\mathbf{Sch}/X$, and let $Y\overset{g}{\to}\mathbf{SO}(\mathcal{Q_C})$ and $Y\overset{h}{\to}\mathbf{S}_\mathcal{Q_C}^2$ be $X$–morphisms such that $\Pi_+g=\rho h$. Then
    \[
    (\Pi_+|_{U_i})(g|_{U_i})=(\Pi_+g)|_{U_i}=(\rho h)|_{U_i}=(\rho|_{U_i})(h|_{U_i}),
    \]
    so for all $i\in I$ there exists a morphism $\lambda_i:Y|_{U_i}\to\mathbf{RT}(\mathcal{C})|_{U_i}$ such that
    \[
    g|_{U_i}=f_1|_{U_i}\lambda_i\text{ and } h|_{U_i}=\Pi|_{U_i}\lambda_i.
    \]
    On the intersection $U_i\cap U_j$, both $\lambda_i$ and $\lambda_j$ make the restricted diagram commute. By the uniqueness in a Cartesian square, we have
    \[
    \lambda_i|_{U_i\cap U_j}=\lambda_j|_{U_i\cap U_j}.
    \]
    Then there exists a $\lambda:Y\to\mathbf{RT}(\mathcal{C})$ such that $\lambda|_{U_i}=\lambda_i$ \cite[Lemma 6.33.1]{stacks}. In particular, this $\lambda$ satisfies
    \[
    g=f_1\lambda\text{ and }h=\Pi\lambda
    \]
    and we are done.
\end{proof}
We have two projections $p_i:\mathbf{S}_\mathcal{Q_C}^2/\boldsymbol{\mu}_2\to\mathbf{S}_\mathcal{Q_C}/\boldsymbol{\mu}_2$, $i=1,2$, defined by projection onto the first and second coordinate respectively. We get two actions of $\mathbf{RT}(\mathcal{C})$ on $\mathbf{S}_\mathcal{Q_C}$:
\[
\big((t_1,t_2,t_3),u\big)\mapsto t_3(u),\quad\text{and }\quad\big((t_1,t_2,t_3),v\big)\mapsto t_2(v).
\]
This induces two actions of $\mathbf{RT}(\mathcal{C})/\boldsymbol{\mu}_2\simeq\mathbf{SO}(\mathcal{Q_C})$ on $\mathbf{S}_\mathcal{Q_C}/\boldsymbol{\mu}_2$, which we denote by $(g,x)\mapsto g\bullet_ix$, $i=1,2$.

\begin{lemma}\label{lemma4.17}
    Both actions of $\mathbf{SO}(\mathcal{Q_C})(X)$ on $(\mathbf{S}_{\mathcal{Q_C}}/\boldsymbol{\mu}_2)(X)$ are transitive.
\end{lemma}
\begin{proof}
    First we will prove that the orbit map $\mathbf{SO}(\mathcal{Q_C})\to\mathbf{S}_\mathcal{Q_C}/\boldsymbol{\mu}_2$, $g\mapsto g\bullet_j[1]$ admits a splitting for each $j=1,2$. Let $j=2$. We have a map $\mathbf{S}_\mathcal{Q_C}\to\mathbf{SO}(\mathcal{Q_C})$ defined by
    \begin{align*}
        \mathbf{S}_\mathcal{Q_C}(Y)&\to\mathbf{SO}(\mathcal{Q_C})(Y)\\
        a&\mapsto B_a
    \end{align*}
    The map is $\boldsymbol{\mu}_2$–invariant, since for any $X$–scheme $Y$ and any $\eta\in\boldsymbol{\mu}_2(Y)$ 
    \[
    B_{\eta a}(Y)(x)=\eta a\cdot x\cdot\eta a=\eta^2\cdot a\cdot x\cdot a=a\cdot x\cdot a=B_a(V)(x),
    \]
    so it induces a map $h:\mathbf{S}_\mathcal{Q_C}/\boldsymbol{\mu}_2\to\mathbf{SO}(\mathcal{Q_C})$. Let $x\in(\mathbf{S}_\mathcal{Q_C}/\boldsymbol{\mu}_2)(X)$ and $\{U_i\to X\}$ be an fppf covering. Then for every $i$, $x_{U_i}$ lifts to some $a_i\in\mathbf{S}_\mathcal{Q_C}(U_i)$. Then $h(X)(x)_{U_i}=B_{a_i}(U_i)$, which lifts to $(B_{a_i},R_{\overline{a_i}},L_{\overline{a_i}})=:\mathbf{t}\in\mathbf{RT}(\mathcal{C})(U_i)$ (see example \ref{ex5.2}). We have
    \[
    h(X)(x)_{U_i}\bullet_2x_{U_i}=p_1(\mathbf{t}.x_{U_i})=R_{\overline{a_i}}(U_i)[a_i]=[R_{\overline{a_i}}(U_i)(a_i)]=[1],
    \]
    and since it holds for each $U_i$ it follows that 
    \[
    h(X)(x)^{-1}\bullet_2[1]=x.
    \]
    So $x\mapsto h(X)(x)^{-1}$ is a section of the orbit map. The case $j=1$ is done similarly, but with the first projection $p_1$. Now, let $x,y\in \mathbf{S}_\mathcal{Q_C}/\boldsymbol{\mu}_2$. We have $h(X)(y)^{-1}\cdot h(X)(x)\in\mathbf{SO}(\mathcal{Q_C})(X)$, so
    \[
    (h(X)(y)^{-1}\cdot h(X)(x))\bullet_jx=h(X)(y)^{-1}\bullet_j(h(X)(x)\bullet_j x)=h(X)(y)^{-1}\bullet_j[1]=y.
    \]
    Hence the actions are transitive.
\end{proof}

Now we have everything we need to state and prove our last theorem.

\begin{theorem}\label{thm4.17}
    Let $\mathbf{G}=\mathbf{Aut}(\mathcal{C})$ and let
    \begin{align*}
        S_1&=\ker\big(H^1_\text{fppf}(X,\mathbf{G})\to H^1_\text{fppf}(X,\mathbf{RT}(\mathcal{C}))\big),\\
        S_2&=\ker\big(H^1_\text{fppf}(X,\mathbf{G})\to H^1_\text{fppf}(X,\mathbf{O}(\mathcal{Q}_\mathcal{C}))\big),
    \end{align*}
    where the maps are induced by
    \begin{align*}
        \mathbf{Aut}(\mathcal{C})&\to\mathbf{RT}(\mathcal{C}),\qquad\qquad\mathbf{Aut}(\mathcal{C})\to\mathbf{O}(\mathcal{Q}_\mathcal{C})\\
        \varphi&\mapsto(\varphi,\varphi,\varphi)\qquad\qquad\qquad\psi\mapsto\psi
    \end{align*}
    Then $S_1=S_2$.
\end{theorem}

\begin{remark}
    The maps are well–defined by remark \ref{remark3.5} and theorem \ref{thm4.10}(1).
\end{remark}

\begin{proof}[Proof of theorem \ref{thm4.17}]
    Parts of this proof is inspired by the proof of theorem 6.6 in \cite{OA}. Notice that $S_1$ classifies the octonion algebras over $X$ that are isomorphic to $\mathcal{C}^{a,b}$. We have homomorphisms 
    \begin{align*}
        i_1:\mathbf{Aut}(\mathcal{C})&\to\mathbf{RT}(\mathcal{C}),\qquad\qquad\qquad i_2:\mathbf{RT}(\mathcal{C})\to\mathbf{SO}(\mathcal{Q}_\mathcal{C})\\
        \varphi&\mapsto(\varphi,\varphi,\varphi)\qquad\qquad\qquad(\psi_{1},\psi_{2},\psi_{3})\mapsto\psi_{1}\\
        i_3:\mathbf{SO}(\mathcal{Q}_\mathcal{C})&\to\mathbf{O}(\mathcal{Q}_\mathcal{C}),\qquad\quad i_4:\mathbf{Aut}(\mathcal{C})\to\mathbf{O}(\mathcal{Q}_\mathcal{C})\\
        \rho&\mapsto\rho\qquad\qquad\qquad\qquad\qquad\varphi\mapsto\varphi\\
    \end{align*}
    We see that $i_4=i_3\circ i_2\circ i_1$. $H^1(X,\_)$ is a covariant functor so $i_4^\ast=i_3^\ast\circ i_2^\ast\circ i_1^\ast$, where $i_j^\ast:=H^1_\text{fppf}(X,i_j)$. From this it follows that $S_1\subseteq S_2$. 
    
    Let $[C']\in S_2$. Then there exists $\varphi:\mathcal{C}\to C'$ such that $\varphi\in\mathbf{O}(\mathcal{Q_C})$. For any affine connected subset $U\subset X$, either $\varphi|_U\in\mathbf{SO}(\mathcal{Q_C})$ or $\varphi|_U\in\mathbf{O}(\mathcal{Q_C})\setminus\mathbf{SO}(\mathcal{Q_C})$ \cite[Theorem 6.6]{OA}. Let $X_i$ be the connected components of $X$, and $\{U_{ij}\}$ an affine connected cover of $X_i$. If for some $j_1,j_2$ we have $U_{ij_1}\cap U_{ij_2}\neq\varnothing$ and
    \[
    \varphi|_{U_{ij_1}}\in\mathbf{SO}(\mathcal{Q_C}),\text{ and }\varphi|_{U_{ij_2}}\in\mathbf{O}(\mathcal{Q_C})\setminus\mathbf{SO}(\mathcal{Q_C}),
    \]
    then there exists an open affine $V\subset U_{ij_1}\cap U_{ij_2}$ such that
    \[
    \varphi|_V\in\mathbf{SO}(\mathcal{Q_C}),\text{ and }\varphi|_V\in\mathbf{O}(\mathcal{Q_C})\setminus\mathbf{SO}(\mathcal{Q_C}),
    \]
    a contradiction. Hence, for every $i$, either
    \[
    \varphi|_{U_{ij}}\in\mathbf{SO}(\mathcal{Q_C})\quad\forall j
    \]
    or
    \[
    \varphi|_{U_{ij}}\in\mathbf{O}(\mathcal{Q_C})\setminus\mathbf{SO}(\mathcal{Q_C})\quad\forall j.
    \]
    Then there exists a morphism $\psi:\mathcal{C}\to C'$ such that
    \[
    \psi|_{U_{ij}}=\left\{
    \begin{array}{cc}
         \varphi|_{U_{ij}},&\text{if }\varphi|_{U_{ij}}\in\mathbf{SO}(\mathcal{Q_C})\\
         \kappa|_{U_{ij}}\circ\varphi|_{U_{ij}},&\text{if }\varphi|_{U_{ij}}\in\mathbf{O}(\mathcal{Q_C})\setminus\mathbf{SO}(\mathcal{Q_C})
    \end{array}
    \right.
    \]
    Then $\psi\in\mathbf{SO}(\mathcal{Q_C})$. There exists a bijection \cite[Proposition 2.4.3]{Gille}
    \begin{align*}
        \phi:\mathbf{SO}(\mathcal{Q_C})(X)\backslash(\mathbf{S}_\mathcal{Q_C}^2/\boldsymbol{\mu}_2)(X)&\overset{\sim}{\longrightarrow}S_2\\
        x&\longmapsto\Pi_+^{-1}(x)\wedge^\mathbf{G}\mathbf{W}(\mathcal{C}),
    \end{align*}
    where $\mathbf{SO}(\mathcal{Q_C})(X)\backslash(\mathbf{S}_\mathcal{Q_C}^2/\boldsymbol{\mu}_2)(X)$ is the set of orbits of the $\mathbf{SO}(\mathcal{Q_C})(X)$–action. Each orbit can be represented by an element in $(\mathbf{S}_\mathcal{Q_C}^2/\boldsymbol{\mu}_2)(X)$, so let $x\in(\mathbf{S}_\mathcal{Q_C}^2/\boldsymbol{\mu}_2)(X)$ such that $\phi(x)=[C']$ and let $x_1$ be its projection onto the first copy of $(\mathbf{S}_\mathcal{Q_C}/\boldsymbol{\mu}_2)(X)$. By Lemma \ref{lemma4.17}, we may assume $x_1=[1]\in(\mathbf{S}_\mathcal{Q_C}/\boldsymbol{\mu}_2)(X)$. Consider the commutative diagram
    \[
    \begin{tikzcd}
        \mathbf{S}_\mathcal{Q_C}^2\arrow[r]\arrow[d,"\varrho_1"]&\mathbf{S}_\mathcal{Q_C}\arrow[d,"\varrho_2"]\\
        \mathbf{S}_\mathcal{Q_C}^2/\boldsymbol{\mu}_2\arrow[r]&\mathbf{S}_\mathcal{Q_C}/\boldsymbol{\mu}_2
    \end{tikzcd}
    \]
    where the horizontal maps are projections onto the first coordinate. Note that $\varrho_1$ and $\varrho_2$ define $\boldsymbol{\mu}_2$–torsors over $\mathbf{S}_\mathcal{Q_C}^2/\boldsymbol{\mu}_2$ and $\mathbf{S}_\mathcal{Q_C}/\boldsymbol{\mu}_2$ respectively. By the commutativity of the diagram, we get an isomorphism $\varrho_1^{-1}(x)\to\varrho_2^{-1}([1])$ of $\boldsymbol{\mu}_2$–torsors. Since $\varrho_2^{-1}(X)([1])=\boldsymbol{\mu}_2(X)\cdot1\neq\varnothing$, we have $\varrho_1^{-1}(x)\neq\varnothing$, so $x$ lifts to an element $(a,b)\in\mathbf{S}_\mathcal{Q_C}^2(X)$. By the Cartesian diagram $($\ref{cartdiag}$)$ we have a morphism $\Pi^{-1}(a,b)\overset{f_1}{\to}\Pi_+^{-1}(x)$ and, since $\Pi$ and $\Pi_+$ are $\mathbf{G}$–torsors, this morphism must be an isomorphism. Then we have an isomorphism of fppf–sheaves of algebras
    \[
    \mathbf{W}(\mathcal{C}^{a,b})\simeq\Pi^{-1}(a,b)\wedge^\mathbf{G}\mathbf{W}(\mathcal{C})\simeq\Pi_+^{-1}(x)\wedge^\mathbf{G}\mathbf{W}(\mathcal{C})\simeq\mathbf{W}(C').
    \]
    Thus $[C']\in S_1$.
\end{proof}

\begin{corollary}
    Let $\mathcal{C}$ and $\mathcal{C}'$ be two octonion algebras over X. The quadratic forms $\mathcal{Q}_\mathcal{C}$ and $\mathcal{Q}_{\mathcal{C}'}$ are isometric if and only if there exists $a,b\in\mathbf{S}_\mathcal{Q_C}(X)$ such that $\mathcal{C}'$ is isomorphic to $\mathcal{C}^{a,b}$.
\end{corollary}
\begin{proof}
    By theorem \ref{thm4.7} and lemma \ref{lemma4.9}, the set $S_1$ classifies the octonion algebras over $X$ isomorphic to $\mathcal{C}^{a,b}$ for some $a,b\in\mathbf{S}_\mathcal{Q_C}(X)$. Also by lemma \ref{lemma4.9}, the set $S_2$ classifies octonion algebras whose norm is isometric to $\mathcal{Q_C}$. By the equality established in theorem \ref{thm4.17}, the statement follows.
\end{proof}

\bibliographystyle{siam}
\bibliography{ref.bib}

\end{document}